\documentclass{amsart}
\usepackage{graphicx}
\usepackage[all]{xy}
\usepackage[numbers]{natbib}
\usepackage[latin1]{inputenc}
\usepackage{amsmath,url,amssymb,latexsym,hyperref}
\usepackage[final]{showkeys}
\usepackage[T1]{fontenc}
\usepackage{color}
\theoremstyle{definition}
\newtheorem{mydef}{Definition}[section]

\newtheorem{lemma}[mydef]{Lemma}
\newtheorem{theorem}[mydef]{Theorem}

\newtheorem{proposition}[mydef]{Proposition}
\newtheorem{definition}[mydef]{Definition}
\newtheorem{example}[mydef]{Example}

\newtheorem{remark}[mydef]{Remark}

\newtheorem{note}[mydef]{Note}

\def\monster{\mathbb{M}}

\def\E{\varepsilon}
\newcommand{\cf}[1]{\text{cf}(#1)}
\def\dist{\mathbf{d}}
\def\V{\mathbb{V}}

\def\ordencea{\prec_{\C{K}}}

\def\ck{\C{K}}
\def\cl{\C{L}}
\def\lsk{\operatorname{LS}(\ck)}
\def\cmet{{\bf cMet}}

\def\abel{{\bf Ab}}
\def\fabel{{\bf FrAb}}

\def\grp{{\bf Grp}}
\def\sets{{\bf Set}}
\newcommand\dotminus{%
  \ooalign{\hidewidth\raise1ex\hbox{.}\hidewidth\cr$\;-\;$\cr}%
}
\newcommand{\C}[1]{{\mathcal #1}}
\newcommand{\shcat}[1]{{\bf Sh}_#1}
\newcommand{\dc}[1]{\operatorname{dc}(#1)}
\newcommand{\metr}[1]{{\bf Met}_#1}
\newcommand{\metrstr}[2]{{\bf MetStr}_#1(#2)}
\newcommand{\psmetr}[1]{{\bf PsMet}_#1}
\newcommand{\psmetrstr}[2]{{\bf PsMetStr}_#1(#2)}
\newcommand{\metrc}[1]{{\bf Met}_#1^{\le}}
\newcommand{\psmetrc}[1]{{\bf PsMet}_#1^{\le}}
\def\homk{\operatorname{Hom}_\ck}
\def\restr{\upharpoonright}
\newcommand{\alglatt}[1][V]{\mathbb{#1}=\langle #1, +, 0, \le  \rangle}

\newcommand{\AEC}[2][K]{\langle \mathcal{#1}_{#2},\prec_{\mathcal{#1}_{#2}}  \rangle}

\mathchardef\mhyphen="2D
\newcommand{\gatp}{{\sf ga\mhyphen tp}}
\newcommand{\gaS}{{\sf ga\mhyphen S}}

\begin{document}

\title{Tameness in generalized metric structures}

\author[Lieberman]{Michael Lieberman}
\email{qmlieberman@vutbr.cz}
\urladdr{https://math.fme.vutbr.cz/Home/lieberman}
\address{Institute of Mathematics, Faculty of Mechanical Engineering, Brno University of Technology, Brno, Czech Republic}

\author[Rosick\'y]{Ji\v r\'i Rosick\'y}
\email{rosicky@math.muni.cz}
\urladdr{http://www.math.muni.cz/\textasciitilde rosicky/}
\address{Department of Mathematics and Statistics, Faculty of Science, Masaryk University, Brno, Czech Republic}
\thanks{The first and second authors were supported by the Grant Agency of the Czech Republic under the grant P201/12/G028. The third author was supported by Universidad Nacional de Colombia under the grants 41705 and 48359.}

\author[Zambrano]{Pedro Zambrano}
\email{phzambranor@unal.edu.co, phzambranor@gmail.com}
\address{Departamento de Matem\'aticas, Universidad Nacional de Colombia, AK 30 $\#$ 45-03 c\'odigo postal 111321, Bogota, Colombia.}

\maketitle


\begin{abstract}
We broaden the framework of metric abstract elementary classes (mAECs) in several essential ways, chiefly by allowing the metric to take values in a well-behaved quantale.  As a proof of concept we show that the result of \cite{boneyzam} on (metric) tameness under a large cardinal assumption holds in this more general context.  We briefly consider a further generalization to {\it partial metric spaces}, and hint at connections to classes of fuzzy structures, and structures on sheaves.
\end{abstract}

\section{Introduction}

This paper lies in the rapidly developing intersection of abstract model theory, large cardinals, and category theory.  On a technical level, we concern ourselves with generalizations of the phenomenon of tameness---an essential condition that ensures the behavior of types is determined by their restrictions to subsets of their domain of a (uniform) small size---and its derivability from a suitable large cardinal assumption using the method of accessible images (see e.g. \cite{LiRo17}).  The generalization itself, from abstract classes of complete metric structures to abstract classes of structures over quantale-valued spaces, is relatively straightforward: the chief interest, it is to be hoped, lies in the new connections that it affords.  In particular, the new framework, which we refer to as {\it $\V$-abstract elementary classes}, or {\it $\V$-AECs}, encompasses new classes of examples, including probabilistic metric structures.  A further generalization to partial metric spaces (those in which self-distances need not be $0$) should lead to connections with abstract classes of fuzzy structures and sheaves, thereby broadening the applicability of a host of results and methods from abstract model theory to new areas, e.g. those describable in Lukasiewicz logic, rather than first-order or continuous logic.  

We here present only a brief introduction to the abstract model-theoretic framework that serves as our inspiration.  The origins of the current work lie in the notion of an {\it abstract elementary class}, or {\it AEC}, which was introduced by Shelah (\cite{shelaec}) as a framework in which to consider certain classes of mathematical structures that cannot be satisfactorily axiomatized in finitary first order logic, e.g. Artinian rings, or the complex numbers with exponentiation.  A comprehensive treatment of AECs can be found in \cite{baldwinbk}, and updates on the state-of-the-art in \cite{bv-survey-bfo} and \cite{shvaseycatmult}.  In essence, an AEC is a purely category-theoretic generalization of an {\it elementary class}---that is, the class of models of a (complete) first order theory---discarding syntax and retaining only the essential properties of elementary embeddings.  The move away from syntax allows us to work simultaneously on the model theory of a wide variety of logics, chiefly those generalized logics of the form $L_{\kappa,\omega}$ and $L(Q)$.  The allusion to category theory above is not a superficial one, incidentally: \cite{catthaspects} and \cite{bekerosicky} give, independently, characterizations of AECs as particular kinds of accessible categories (see Section~\ref{subseccatprelims} below).

The shift to AECs requires us to embrace a new, nonsyntactic notion of type; that is, Shelah's {\it Galois types} (or {\it orbital types}).  In short, given a model $M$, embeddings $f_i:M\to M_i$, $i=0,1$, and elements $a_i\in |M_i|$, we say that $(f_0,a_0)$ and $(f_1,a_1)$ are equivalent if there is a commutative square
$$\xymatrix@=3pc{
        M_1 \ar[r]^{g_1} & \bar{M} \\
        M \ar [u]^{f_1} \ar [r]_{f_0} &
        M_0 \ar[u]_{g_0}
      }$$
such that $g_0(a_0)=g_1(a_1)$.  The set of {\it Galois types over $M$} is precisely the set of equivalence classes of pairs $(f,a)$, with $f:M\to M'$.  Under the simplifying assumption of the existence of a monster model, i.e. a large, saturated, model-homogeneous model $\monster$, the Galois types over a model $M$ can be identified with the orbits of elements of $\monster$ under automorphisms of $\monster$ fixing $M$---that is, two elements of $\monster$ satisfy the same type if and only if they are indistinguishable up to the information contained in $M$.  To say that an AEC is $\chi$-tame, $\chi$ an infinite cardinal, is to say that if a model $M$ contains sufficient information to distinguish $a_0,a_1\in |\monster|$, then there is a submodel $N\ordencea M$ with cardinality less than or equal to $\chi$ that contains enough information to distinguish them.

This ability to reduce questions of equivalence of types to ones over models of a uniform small size has proven to be an essential ingredient in extant results on stability and categoricity transfer in AECs---a {\it dividing line}, in the sense of Shelah, between structure and nonstructure---see, e.g. \cite{bkv}, \cite{grva}, \cite{grvaonesucc}, \cite{bv-survey-bfo}.  Recently, Boney has shown (\cite{boneylarge}) that under the assumption of a proper class of strongly compact cardinals, every AEC is tame; that is, $\chi$-tame for some $\chi$.  This led to a flurry of work on tameness and large cardinals, culminating in \cite{boun}, where it is shown that tameness of AECs is in fact equivalent to the existence of a proper class of almost strongly compact cardinals.  Category theory makes a surprise appearance in this flurry, with the first and second authors noting in \cite{LRclass} that an old result of \cite{makkai-pare}---that the suitably closed image of any accessible functor is accessible assuming a proper class of strongly compact cardinals (improved to almost strongly compact cardinals in \cite{BT-R})---provides an alternate proof of Boney's Theorem.  This will form the backbone of our own result along these lines, in Section~\ref{sectame}.

The abstract classes of metric structures that we will consider here arise from a parallel development, spurred by the attempts in the `60s and `70s by Henson, Chang and Keisler, and others to develop a satisfactory first order theory of metric structures in general, and Banach spaces in particular. In fact, Shelah and Stern have shown that the full first-order theory of Banach spaces has the same essential behavior as second-order logic with quantification over countable sets, \cite{ShSt}. This means, in particular, that its Hanf number is above the least measurable cardinal---if there is one---which is utterly disqualifying.  Ways around this difficulty included restricting the set of formulas, resulting in the positive bounded logic of, e.g. \cite{hens}.  More significantly for us, it can also be resolved through the program of {\it continuous logic} introduced in \cite{changkeis}.  While continuous-valued logic makes it possible to analyze classes of Banach spaces with a toolkit similar to that of first-order logic (for example, continuous logic satisfies L\"owenheim-Skolem-Tarski and compactness theorems), there are some examples of classes of Hilbert spaces which are not axiomatizable in this logic (for example, Hilbert spaces with unbounded self-adjoint closed operators, \cite{Argoty}). This forces a further AEC-style generalization, and ultimately leads to the notion of a {\it metric AEC} (or {\it mAEC}), cf. \cite{hirhyt} or \cite{Za11}, which closely resembles an AEC, but with the following important modifications:
\begin{itemize}
\item The objects of an mAEC have underlying complete metric spaces, rather than discrete sets, and the interpretations of function and relation symbols in the ambient language must be suitably continuous.
\item The class need not be closed under unions of chains---the union of a chain of complete metric spaces need not be complete---but does contain the completion of such unions.  Put another way, an mAEC has directed colimits, but they need not be concrete (\cite{lrmcaec}).
\item In the L\"owenheim-Skolem axiom, cardinality is replaced by {\it density character}.
\end{itemize}
One can again define Galois types as orbits in a monster model $\monster$, although now $\monster$ is itself a metric space: the set of types over a model $M$ comes with a natural pseudometric, given by the Hausdorff distance between the corresponding orbits in $\monster$.  It is customary to make the additional technical assumption of the {\it continuity of types property} (or {\it CTP}, \cite{Za11}; also known in \cite{hirhyt} as the {\it perturbation property}), in which case this notion of distance is in fact a metric.  (Note: in \cite[5.2]{LiRo17} and in Definition~\ref{deftame} below, it is better to work without reference to the monster model, necessitating an alternative, but equivalent, definition of this distance.) This makes possible an $\epsilon-\delta$ reformulation of tameness---{\it {\bf d}-tameness}: for every $\epsilon>0$, there is $\delta>0$ such that if any types $p,q$ over $M$ are separated by more than $\epsilon$, their restrictions to some small $N\ordencea M$ are separated by more than $\delta$.  As in AECs, a stability transfer theorem in mAECs holds under $\mathbf{d}$-tameness, \cite{Za12}.  In \cite{boneyzam}, $\mathbf{d}$-tameness is shown to hold for all mAECs under the assumption of a proper class of strongly compact cardinals, via metric ultrafilters; in \cite{LiRo17}, it is shown to hold for all mAECs with $\epsilon=\delta$, and under the assumption of a proper class of almost strongly compact cardinals.  The latter again reduces the problem to the accessibility of images of accessible functors, slightly generalizing the argument from the discrete case, and providing the template for the still more general tameness result of Section~\ref{sectame}.

Indeed, the primary motivation for the present study is the desire to find the most general level at which this method of argument will work.  Careful examination reveals that one can toy as much as one likes with the axioms governing the metric spaces underlying the structures in the class---nothing fails, for example, if we consider classes of pseudometric structures, partial metric structures (those in which self-distances need not be zero, Section~\ref{secpartial}) or extended (that is, $({\mathbb R}\cup\{\infty\})$-valued) metric structures (\cite{lawveremetr}).  The latter suggests a much more interesting direction: instead of simply dropping metric axioms, can we replace ${\mathbb R}\cup\{\infty\}$ with an arbitrary {\it quantale} $\V$?  What properties of the quantale ${\mathbb R}\cup\{\infty\}$ are indispensable in this context?  As it happens, we require the quantale to be \emph{strongly Flagg} (see Definition~\ref{defsflagg}), a strengthening of the more familiar notion of continuity.  

This is of independent interest, but also points to important future applications: in Section~\ref{secpartial} we briefly consider structures over partial metric spaces, and restrict to the case where $\V=\Omega$ is a continuous frame.  There we suggest the following equivalences of categories (following \cite{partcors}, \cite{fuzzysetsi}, and \cite{fuzzysetsii}):

$\begin{array}{rcl} \text{complete partial }\Omega\text{-valued metrics} & \simeq & \text{complete }\Omega\text{-sets}\\
& \simeq & \text{sheaves on }\Omega\\
 & \simeq & \text{fuzzy sets over }\Omega\end{array}$
 
{\noindent}This means it should be possible to extend our result, almost for free, to the fuzzy and sheafy contexts---in either case, this would represent the first development of tameness and its connection to large cardinals.  We provide only the outline of this project here, leaving the details for future work.

The authors gratefully acknowledge the comments of the anonymous referee, which led to substantial improvements to this paper.

\section{Preliminaries}\label{secprelims}

\subsection{Accessible categories, accessible images}\label{subseccatprelims}

We here provide a few of the essential definitions concerning accessible categories---with examples---along with the result of \cite{BT-R} concerning the accessibility of powerful images that we will require in Section~\ref{sectame}.  For further details on accessible categories, see \cite{adros} or \cite{makkai-pare}.  The method of argument by accessibility of powerful images is explained and thoroughly illustrated in \cite{LiRo17}.  Throughout, we will assume a basic familiarity with category-theoretic terminology, along the lines of \cite{absconcrcats}.

As a basic intuition, an accessible category is one which is closed under sufficiently directed colimits (otherwise known, in many areas of mathematics, as {\it direct limits}), and where any object can be built---via a highly directed colimit---from a set of ``small'' objects.  To be more precise:

\begin{definition}\label{defacc}
	Let $\ck$ be a category, $\lambda$ an infinite regular cardinal.
	\begin{enumerate}
	\item We say that a poset $I$ is {\it $\lambda$-directed} if for any $A\subseteq I$ with $|A|<\lambda$, there is $i\in I$ such that $i\geq a$ for all $a\in A$.  A {\it $\lambda$-directed colimit} in $\ck$ is a colimit whose underlying diagram, $D:I\to \ck$, has $I$ $\lambda$-directed.
	\item An object $N$ in $\ck$ is $\lambda$-presentable if the associated hom-functor, 
	$$\homk(N,-):\ck\to\sets,$$
	preserves $\lambda$-directed colimits.  More transparently, $N$ is $\lambda$-presentable if for any $\lambda$-directed colimit $(\phi_i:M_i\to M\mid i\in I$) in $\ck$, any morphism $f:N\to M$ factors essentially uniquely through one of the the $M_i$; that is, $f=\phi_i\circ f_i$ for some $i\in I$.
	\item We say $\ck$ is {\it $\lambda$-accessible} if it
		\begin{itemize}
		\item has all $\lambda$-directed colimits,
		\item contains a set (up to isomorphism) of $\lambda$-presentable objects,
		\item and any object is a $\lambda$-directed colimit of $\lambda$-presentable objects.
		\end{itemize}
	We say $\ck$ is {\it accessible} if it is $\lambda$-accessible for some regular $\lambda$.
	\end{enumerate}
\end{definition}

\begin{example}\label{accexams}\mbox{ }
\begin{enumerate}
	\item {\bf $\lambda$-directed colimits:}
	\begin{enumerate}
		\item A general category $\ck$ has all directed colimits if and only if it has all colimits of chains---the former are reducible to the latter (e.g. \cite[1.7]{adros}).	  
		\item Following the suggestion concerning mAECs in the introduction above, the category $\cmet$ of complete metric spaces and contractions has directed colimits, but only the $\aleph_1$-directed colimits are concrete: if we take the union of an $\omega$-chain of complete spaces, the result need not be complete (\cite[4.5(3)]{lrmcaec}).
		\item For similar reasons, given an AEC $\ck$, the subcategory of $\lambda$-(Galois)-saturated objects will not generally have $\mu$-directed colimits when $\cf{\lambda}>\mu$.  In general, this category will contain colimits of short chains only under additional assumptions on $\ck$, e.g. tameness or stability (\cite{bvchains}). 
	\end{enumerate}
	\item {\bf Presentability and size:}
		\begin{enumerate}
		\item In $\grp$, the category of groups, an object is $\omega$-presentable (more commonly, {\it finitely presentable}) if and only if it is finitely presented in the classical sense.  In fact, in a general variety of finitary algebras with fewer than $\lambda$ operations, an algebra is $\lambda$-presentable if and only if it is $\lambda$-presented: that is, it is generated by fewer than $\lambda$ elements, subject to fewer than $\lambda$ equations \cite[3.12]{adros}.
		\item In $\cmet$, an object $M$ is $\lambda$-presentable if and only if its density character, $\dc{M}$, is less than $\lambda$, provided $\lambda$ is uncountable---recall that the density character of a complete space is the least cardinality of a dense subset.  $\cmet$ contains no $\omega$-presentable objects \cite[3.5]{lrmcaec}.
		\item More broadly, in any accessible category with directed colimits, the smallest cardinal $\lambda$ for which an object $M$ is $\lambda$-presentable---its {\it presentability rank}---is always a successor cardinal, say $\mu^+$(see \cite{bekerosicky}), so we may define the {\it internal size} of $M$ to be $\mu$.  This means, in specific cases, precisely what we would like: if $\ck$ is an AEC, the internal size of $M$ is its cardinality; if $\ck$ is an mAEC, the internal size of $M$ is its density character \cite[3.4]{lrmcaec}.\end{enumerate}
	\item {\bf Accessibility:}
	\begin{enumerate}\item The category $\grp$ is $\omega$-accessible (or {\it finitely accessible}): it has directed colimits, and any group can be expressed as the directed colimit of its finitely presented subgroups (which are finitely presentable, by \ref{accexams}(2)(a) above).  Similarly, any variety of finitary algebras with fewer than $\lambda$ operations is $\lambda$-accessible, for $\lambda$ a regular cardinal.
	\item The category $\cmet$ is $\aleph_1$-accessible, but not finitely accessible, in part by \ref{accexams}(1)(b) and \ref{accexams}(2)(b).  It is not elementary to see that a space can be obtained as an $\aleph_1$-directed colimit of its separable subspaces: an alternative argument can be given (see \cite[4.5(3)]{lrmcaec}).
	\item If $\ck$ is an AEC, closure under unions of chains---hence directed colimits, per \ref{accexams}(1)(a)---and the L\"owenheim-Skolem property ensure that $\ck$ is $\lambda$-accessible for all regular $\lambda>\lsk$, where $\lsk$ is the L\"owenheim-Skolem number of $\ck$.  The same holds for any mAEC, by the metric variants of the aforementioned axioms.\end{enumerate}
\end{enumerate}
\end{example}

We need a few terms associated with functors, as well:

\begin{definition}\label{functdefs}
	\begin{enumerate}\item 	We say that a functor $F:\ck\to\cl$ is {\it $\lambda$-accessible}, $\lambda$ regular, if $\ck$ and $\cl$ are $\lambda$-accessible and $F$ preserves $\lambda$-directed colimits.  We say $F$ is {\it accessible} if it is $\lambda$-accessible for some $\lambda$.
	\item Given a functor $F:\ck\to\cl$, the {\it powerful image} of $F$ is the closure of the image of $F$ under $\cl$-subobjects.\end{enumerate}
\end{definition}

\begin{example}\label{freeabexample}
Consider the free Abelian group functor, $F:\sets\to\abel$, which assigns to each set $X$ the free Abelian group generated by the elements of $X$.  This functor is finitely accessible.  The image of $F$ is the category of free abelian groups, $\fabel$---it is already closed under subobjects in $\abel$ (since any subgroup of a free abelian group is free Abelian), so it is in fact the powerful image of $F$.
\end{example}

This example is instructive in the following sense: the question of whether $\fabel$, the powerful image of $F:\sets\to\abel$, is accessible is highly dependent on the ambient set theory.  Per  \cite{eklofmek}, under V=L, it is not accessible; under the assumption of a strongly compact cardinal, it is.  The question of whether the same is true with respect to the powerful images of arbitrary accessible functors led to the result \cite[5.5.1]{makkai-pare}: if there are arbitrarily large strongly compact cardinals, the powerful image of any accessible functor is accessible.  We will use a recent refinement of this result, which is essentially \cite[3.4]{BT-R}:

\begin{theorem}\label{BTRTh}
Let $\lambda$ be a regular cardinal and $\C{L}$ an accessible category such that there exists a $L_{\mu_{\C{L}},\omega}$-compact cardinal $\kappa$. Suppose, moreover, that $\kappa \trianglerighteq \lambda$.  Then the powerful image of any $\lambda$-accessible functor $F:\C{K}\to\C{L}$ that preserves $\mu_{\C{L}}$-presentable objects is $\kappa$-accessible.
\end{theorem}

In fact, this is a minor refinement of the \cite{BT-R} result, \cite[2.6]{LiRo17}.  Here $\mu_\cl$ is a cardinal computed from the structure of $\cl$---we refer readers to \cite[3.1]{BT-R} for details---and $\trianglerighteq$ denotes the {\it sharp inequality} of \cite{makkai-pare}.  For a rough intuition concerning the latter, note that, per \cite[2.5]{internal-sizes-v3}, if $\kappa>2^{<\lambda}$, $\kappa\triangleright\lambda$ if and only if $\kappa$ is $\lambda$-closed; that is, $\theta^{<\lambda}<\kappa$ for all $\theta<\kappa$.  We recall:

\begin{definition} A cardinal $\kappa\geq\mu$ is {\it $L_{\mu,\omega}$-compact} if any $\kappa$-complete filter can be extended to a $\mu$-complete ultrafilter.  We say $\kappa$ is {\it almost strongly compact} if it is $L_{\mu,\omega}$-compact for all $\mu<\kappa$.  It is {\it strongly compact} if this holds for all $\mu\leq\kappa$.\end{definition}

We adopt the terminology ``$L_{\mu,\omega}$-compact'' from \cite{BT-R}, as a way of sidestepping the multiple conflicting interpretations now associated with the original terminology: ``$\mu$-strongly compact'' (see e.g. \cite{apter}, \cite{menas}, and, more recently, \cite[4.6]{BaMa}).  In any case, at least in principle, the tameness result we derive from Theorem~\ref{BTRTh} in Section~\ref{sectame} will involve a large cardinal assumption weaker than strong compactness (see \cite[2.2]{boun} and \cite[2.4]{BT-R} for technical discussions related to this point).

\subsection{Quantales}\label{subsecquantprelims}

Quantales were first introduced in \cite{mulvquant} as a noncommutative generalization of {\it frames}; that is, complete lattices $\langle L,\leq\rangle$ in which the meet and join operations satisfy the following distributivity condition: for any $x\in L$, and $A\subseteq L$,
$$x\wedge \bigvee_{a\in A} a=\bigvee_{a\in A}(x\wedge a)$$
Mulvey's idea was to replace $\wedge$ with a (possibly noncommutative) binary operation $\star$ satisfying the same distributivity condition.  While his motivation was largely topological, our own route is rather different, and considerably more concrete: for us, quantales are ordered algebraic structures that can stand in for the extended nonnegative reals, $\langle [0,\infty],+,0,\le\rangle$, as the domain of values for a meaningful notion of distance.  

This will mean that, in time, we will wish to work with addition and underlying order dual to the ones that predominate in the category-theoretic literature.  In defining the basic properties we require, however, we will adhere to the more standard formulations.

\begin{definition}\label{defquantale}
	A \emph{commutative, unital quantale} consists of an ordered quadruple $(V,\star,\eta,\leq)$, where $(V,\leq)$ is a complete lattice and has the structure of a commutative monoid under binary operation $\star$, with neutral element $\eta$ and with the added property that $\star$ distributes over arbitrary joins.
\end{definition}

In the sequel, we will simply assume that our quantales are commutative and unital, without making explicit mention of these properties.  It is likely that much of the machinery of this paper can be adapted to, e.g. the noncommutative context, but we leave this for future work.

\begin{remark}\label{rmkadj}
    Note that the quantale operation $\star$ will always have a right adjoint, often denoted by $\mbox{hom}$; that is, for any $x$, $y$, and $z$,
    $$x\star y\leq z\mbox{ if and only if }y\leq\mbox{hom}(x,z).$$
    In additive quantales, this adjoint takes the form of truncated subtraction, $\dotminus$ (see, e.g. Example~\ref{exflaggquant}(2)). 
\end{remark}

Most of the properties that we require relate directly to the underlying lattice, $(V,\leq)$.  Recall:

\begin{definition}\label{deftotbelow}
	Let $(V,\leq)$ be a complete lattice.  Given elements $\alpha,\beta\in V$, we say that $\alpha$ is \emph{totally below} $\beta$ and write
	$$\alpha\lll \beta$$ 
	just in case for every $S\subseteq V$, if $\beta\leq\bigvee S$, then $\alpha\leq \gamma$ for some $\gamma\in S$.
\end{definition}

We note that this relation goes by a number of different names in the literature, including \emph{well-below}, \emph{super-way-below}, and \emph{long-way-below}, and that it is often denoted by $\triangleleft$ (cf. \cite{EGP}).  In any case, it should, in general, be carefully distinguished from the \emph{way below} relation, which holds if the above condition is satisfied only for \emph{directed} subsets $S\subset V$, and is typically denoted by $\ll$.

\begin{example}\label{exbottomprob}
In any complete lattice $\langle V,\leq\rangle$, the relations $\ll$ and $\lll$ will always disagree when it comes to the bottom element, $\bot$: while $\bot\ll\bot$, it is \emph{not} the case that $\bot\lll\bot$. To see this, notice that $\bot=\bigvee\emptyset$.  Since directed sets are necessarily nonempty, the condition for $\bot\ll\bot$ is satisfied trivially; the condition for $\bot\lll\bot$, on the other hand, is obviously false.
\end{example}

\begin{remark}
    One defines the \emph{totally above relation}, $\ggg$, and \emph{way above relation}, $\gg$, dually.  It is important to note that, in general, $\alpha\lll\beta$ is not equivalent to $\beta\ggg\alpha$ (and similarly for $\gg$).  We note, however, that the relation $\ggg$ (respectively, $\gg$) on a complete lattice $\langle V,\leq\rangle$ is precisely the totally below relation (respectively, way below relation) on $\langle V,\geq\rangle$.
\end{remark}

Following \cite{Wood}, we define the following condition, which will be essential in the sequel:

\begin{definition}\label{defconstrcompdist}
	A complete lattice $(V,\leq)$ is \emph{constructively completely distributive} if every $\alpha\in V$ can be represented as
	$$\alpha=\bigvee\{\beta\in V\,|\,\beta\lll \alpha\}.$$
	We say that a quantale is constructively completely distributive if its underlying lattice has this property.
\end{definition}

\begin{remark}\label{rmkccd}
	\begin{enumerate}
	    \item We note that any lattice that is completely distributive in the usual sense is constructively completely distributive; the converse holds under the Axiom of Choice (see \cite[5]{Wood}).
	    \item Constructively complete distributivity is to be distinguished from \emph{continuity}, which is the directed analogue: a complete lattice $(V,\leq)$ is \emph{continuous} if every $\alpha\in V$ can be represented as
	$$\alpha=\bigvee\{\beta\in V\,|\,\beta\ll \alpha\},$$
	where $\ll$ is the way below relation mentioned above.  We return to this point momentarily.
	\end{enumerate}
\end{remark}

\begin{definition}\label{defflaggq}
We say that a quantale $(V,\star,\eta,\leq)$ is a \emph{Flagg quantale} if it is constructively completely distributive and, moreover, for any $\alpha\in V$, $\alpha>\bot$, the set
$$\Downarrow\!\alpha=\{\beta\in V\,|\,\beta\lll \alpha\}$$
is up-directed; that is, $\Downarrow\!\alpha$ is nonempty and any finite subset of $\Downarrow\!\alpha$ has an upper bound in $\Downarrow\!\alpha$.
\end{definition}

\begin{remark}\label{rmkflaggterm} The second condition in the definition above is equivalent to the property that whenever $\alpha\lll\gamma$ and $\beta\lll\gamma$, $\alpha\vee\beta\lll\gamma$.  If we reverse the underlying order and write the quantale additively (that is, with $0$ in place of $\eta$ and $+$ in place of $\star$) we arrive at (a non-directed analogue of) the defining condition of the \emph{value distributive lattices} (or \emph{value quantales}) of Flagg, \cite[1.7]{Fl}.  This motivates the choice of terminology, which originates in \cite{RoTh}.  Flagg's definition also requires that $\bot\ll\eta$: it is clear that in a constructively completely distributive lattice, $\bot\lll\alpha$ for any $\alpha>\bot$, so, in particular, $\bot\lll\eta$.\end{remark}

\begin{proposition}\label{proprelscorresp}
    Let $\V=\langle V,\star,\eta,\leq\rangle$ be a Flagg quantale.  For any $\alpha$ and $\beta$ in $V$, $\beta>\bot$, $\alpha\lll\beta$ if and only if $\alpha\ll\beta$.  That is, the totally below and way below relations agree, except on the pair $(\bot,\bot)$. 
\end{proposition}

\begin{proof}
Note that the requirement $\beta>\bot$ is forced upon us by Example~\ref{exbottomprob}.  It is clear, of course, that $\alpha\ll\beta$ whenever $\alpha\lll\beta$.  To see that the converse is also true, suppose that $\alpha\ll\beta$ and that $\beta\leq\sup S$ for some $S\subset V$.  For simplicity, let $\gamma=\sup S$.  By constructively complete distributivity, 
$$\gamma=\sup\Downarrow\!\gamma.$$
Moreover, the set $\Downarrow\!\gamma$ is up-directed.  By the definition of the way below relation, $\alpha\leq\tau$ for some $\tau\in\Downarrow\!\gamma$; that is, $\tau\lll\gamma=\sup S$.  It follows that $\tau\leq\sigma$ for some $\sigma\in S$.  Hence $\alpha\leq\tau\leq\sigma$, meaning that $\alpha\lll\beta$.
\end{proof}

\begin{remark}\label{rmkflaggcrit} In particular, it follows that a quantale is Flagg just in case it is continuous and $\ll$ and $\lll$ correspond, except at $\bot$.  This criterion is often useful in verifying that familiar quantales are Flagg, as in Examples~\ref{exflaggquant}(1)-(4) below. \end{remark}

\begin{example}\label{exflaggquant}
\begin{enumerate}
\item{\bf Truth values:} Consider $\Omega=\{0,1\}$, with $0<1$.  Here the totally below relation is simply given by $0\lll 1$ and $1\lll 1$.  This becomes a Flagg quantale if we take meet to be the binary operation, and $\eta=1$ the neutral element.  We denote this quantale by $\Omega_{\wedge}$.
\item{\bf Distances:} Consider the extended nonnegative reals $[0,\infty]$ in the reverse of the natural ordering, with extended addition as $\star$---this is, of course, the \emph{Lawvere quantale}, which is inextricably linked to \emph{generalized metric spaces}.  The underlying lattice is known to be continuous, and it is clear that $\ll$ and $\lll$ correspond (in fact, they are precisely the same as the strict order), meaning that this quantale---which we denote by $[0,\infty]_+$---is Flagg, by \ref{rmkflaggcrit}.
\item{\bf Unit interval:} Consider $[0,1]$ with the standard ordering---clearly any quantale structure on this lattice will be Flagg, as above.  Consider, for example, the operation of \emph{\L ukasiewicz addition}:
$$x\oplus y=\max\{a+b-1,0\}.$$ 
As noted in \cite[2.15]{flaggkoppcontsp}, this quantale---the {\it quantale of errors}---is closely connected to Lukasiewicz logic.  We denote this quantale by $[0,1]_\oplus$ (following \cite{clemhof}).

It is worth mentioning that there are several other mathematically significant (Flagg) quantale structures on the unit interval with the standard ordering, namely $[0,1]_\wedge$, with meet as the quantale operation, and $[0,1]_*$, with multiplication as the quantale operation.  The latter is isomorphic to the Lawvere quantale; the former is connected with generalized \emph{ultrametric} spaces.
\item{\bf Dual of the unit interval:} Consider $[0,1]$ with the reverse of the standard ordering and operation given by {\it truncated addition} $\dotplus$: 
$$x\dotplus y=\min(x+y,1).$$  
That this quantale, $[0,1]^{op}_{\dotplus}$, is Flagg follows as for the Lawvere quantale, above.  Alternatively, notice that $[0,1]^{op}_{\dotplus}$ is isomorphic to $[0,1]_\oplus$.

\item{\bf Distance distribution functions:} By a {\it distance distribution function}, we mean a map $F:[0,\infty]\to [0,1]$ that satisfies $F(0)=0$ and is left continuous: for any $x\in [0,\infty]$, $F(x)=\sup_{y<x}F(y)$.  Let ${\Delta}$ denote the set of all distance distribution functions, ordered pointwise.  It is shown in \cite[1.3(3)]{clemhof}, for example, that $\Delta$ is constructively completely distributive as a lattice, with a particularly clean characterization of $\lll$, namely that for any $F,G\in\Delta$, $F\lll G$ if and only if there is $x\in[0,\infty]$ with $F(x)=0$ and $F(\infty)<G(x)$.  With this characterization, it is clear that, moreover, the sets $\Downarrow F$ are up-directed, meaning that any quantale structure on $\Delta$ will be Flagg.

We recall that any quantale operation $\bullet$ on $[0,1]$ determines a quantale operation $\star$ on $\Delta$, defined as follows:
for any $F,G\in\Delta$ and $x\in[0,\infty]$, 
$$(F\star G)(x)=\bigvee_{\alpha+\beta\leq x}F(y)\bullet G(y).$$
We denote the quantale structures arising from $\oplus$, $\wedge$, and $\ast$ in (4) above by, respectively, $\Delta_\oplus$, $\Delta_\wedge$, and $\Delta_\ast$---in a slight abuse of notation, we use the same symbols for the induced operations.  We note that in each case, the neutral element is given by
$$\eta(x)=\left\{\begin{array}{ll} 0 \hspace{5 mm}& \mbox{if }x=0\\ 1 & \mbox{otherwise}\end{array}\right..$$
\item{\bf Frames:} If $\Omega$ is a continuous frame---in particular, the lattice of opens of a locally compact topological space or, indeed, a locally compact locale---then it is, in particular, a Flagg quantale, with meet as the quantale operation.  Provided the space is compact and Hausdorff, the relations $\ll$ and $\lll$ will again coincide (cf. \cite[II.7.2]{picpult}).
\end{enumerate}
\end{example}

\begin{remark}\label{notationll} Since $\ll$ and $\lll$ correspond in any Flagg quantale, we will opt always for the simpler of the two notations in the sequel.\end{remark}

The following is dual to \cite[2.9]{Fl}:

\begin{proposition}\label{epsilonovertwo}
    Let $(V,\star,\eta,\leq)$ be a Flagg quantale.  For every $\epsilon\ll\eta$, there exists $\delta\ll\eta$ with $\epsilon\ll\delta\star\delta$. 
\end{proposition}

\begin{proof}
Since the operation $\star$ preserves $\leq$, it is clear that
$$\bigvee_{\delta\ll\eta}\delta\star\delta\geq\bigvee_{\alpha,\beta\ll\eta}\alpha\star\beta.$$
The latter is equal to
$$\left(\bigvee_{\alpha\ll\eta}\alpha\right)\star\left(\bigvee_{\beta\ll\eta}\beta\right)=\eta\star\eta=\eta.$$
In particular,
$$\bigvee_{\delta\ll\eta}\delta\star\delta\geq\eta$$
and, since $\epsilon\ll\eta$, this implies that for some $\delta\ll\eta$, $\epsilon\leq\delta\star\delta$: in fact, by \cite[1.3]{Fl}, $\delta$ may be chosen so that $\epsilon\ll\delta\star\delta$.
\end{proof}

The significance of this proposition is perhaps clearer if we write the quantale additively and work in the reverse ordering; that is, $(V,+,0,\geq)$.  In this case, Proposition~\ref{epsilonovertwo} says that for any $\epsilon$ with $\epsilon\gg 0$, there is $\delta$ with $0\ll\delta$ with $2\delta\ll\epsilon$---we can, therefore, conduct standard ``epsilon over two'' arguments in this context.

\begin{lemma}\label{Closeness}
Let $(V,\star,\eta,\leq)$ be Flagg. For any $y\in V$, $y=\bigvee \{y\star\E\mid \E\ll \eta \}$.
\end{lemma}
\begin{proof}
Since the quantale is Flagg, $\eta=\bigvee \{\E\in V\mid \E\ll \eta \}$. For any $y\in V$, $y=y\star\eta=y\star\bigvee \{\E\in V\mid \E\ll \eta \}=\bigvee \{y\star\E\mid \E\ll \eta\}$.
\end{proof}

This, again, is more intuitive in the additive case and under the reverse order: for any $y$, $y=\bigwedge\{y+\E\mid \E\gg 0\}$.

For technical reasons, we will occasionally require two additional properties, yielding what we call a \emph{strongly Flagg quantale}.

\begin{definition}\label{defsflagg}Let ${\mathbb V}=\langle V,\star,\eta,\leq\rangle$ be a quantale.
\begin{enumerate}
    \item We say that $\mathbb V$ \emph{has sequential approximation} if there is a sequence $(\alpha_n)_{n\in\omega}$ in $V$ such that
    \begin{itemize}
        \item $\bigvee_{n\in\omega}\alpha_n=\eta$.
        \item For all $n\in\omega$, $\alpha_n\lll\eta$.
        \item For all $n\in\omega$, $\alpha_n\leq \alpha_{n+1}$.
    \end{itemize}
    \item We say that $\mathbb V$ is \emph{totally monotonic} if for any $\alpha,\beta,\gamma\in V$, if $\alpha\lll\beta$, then $\alpha\star \gamma\lll \beta\star\gamma$.
    \item We say that $\mathbb V$ is strongly Flagg if it is Flagg, has sequential approximation and is totally monotonic.
\end{enumerate}
\end{definition}

\begin{example}\label{exsflagg}
\begin{enumerate}
    \item It is trivial to verify that the quantale of truth values $\Omega$ in Example~\ref{exflaggquant}(1) is strongly Flagg: since the neutral element $1$ satisfies $1\ll 1$, the constant sequence suffices to meet condition \ref{defsflagg}(2). 
    
    \item The quantales $[0,\infty]_+$, $[0,1]^{op}_\dotplus$, $[0,1]_{\oplus}$, $[0,1]_{\wedge}$, and $[0,1]_{*}$ clearly have sequential approximation.  All are totally monotonic: as the underlying lattices are chains, $\ll$ corresponds to $<$, and the operations are all monotonic in the sense of $<$.  Hence these quantales are all strongly Flagg.
    
    \item None of the quantales of distance distribution functions described in \ref{exflaggquant}(5) are strongly Flagg: although all appear to have sequential approximation, they are not totally monotonic.  
    
    To see this, we consider step functions of the form
    $$S_{\alpha,u}=\left\{\begin{array}{ll} 0\hspace{5 mm} & x\leq\alpha\\ u & x>\alpha\end{array}\right.$$
    for $\alpha\in[0,\infty]$ and $u\in[0,1]$.  Notice that $S_{\alpha,u}\ll\eta$ for all $\alpha>0$ and $\eta<1$.  
    
    The relation $\ll$ is certainly not preserved by the operation in the quantale $\Delta_*$, for example.  Consider $S_{1,1/2}$ and
    $$F(x)=\left\{\begin{array}{ll}
         0\hspace{10 mm} & x=0 \\
         1/4 & 0<x\leq 1\\
         1 & x>1
    \end{array}\right..$$
    In this case,
    $$(S_{1,1/2}*F)(x)=\bigvee_{\alpha+\beta\leq x} [S_{1,1/2}(\alpha)F(\beta)]=\left\{\begin{array}{ll}
       0\hspace{11 mm}  & x\leq 1\\
       1/8  & 1<x\leq 2\\
       1/2 & x>2
    \end{array}\right.$$
    and it is clear that for any $x\in[0,\infty]$ with $(S_{1,1/2}*F)(x)=0$, we have $F(x)\leq 1/4<1/2=(S_{1,1/2}*F)(\infty)$.  So $S_{1,1/2}*F$ is not totally below $F$, even though $S_{1,1/2}\ll\eta$.
    With a little additional work, the same pair of functions show that $\ll$ is not preserved in $\Delta_\wedge$ or $\Delta_\oplus$, either.
    \item Continuous frames need not be strongly Flagg, as they too will typically fail to be totally monotonic.  Given a compact Hausdorff space $X$, $X$ is locally compact, meaning that its poset of opens, $\mathcal{O}(X)$, is a continuous frame.  It therefore forms a Flagg quantale under the meet operation, with neutral element $\eta=X$.  While compactness of $X$ implies that $\eta\ll\eta$, it is clearly not true in general that $U\ll U$ for arbitrary $U\in\mathcal{O}(X)$: this holds, in particular, only if every $U\in\mathcal{O}(X)$ is compact.
    This difficulty disappears if $X$ is, say, equipped with the finite complement topology---in this case, the quantale of opens is strongly Flagg.
\end{enumerate}
\end{example}

\section{$\V$-metric spaces.}\label{secvmet}

We now develop the theory of metrics with values in Flagg quantales $\V$, with the ultimate aim of speaking of $\V$-metric structures, $\V$-abstract elementary classes, and tameness.  For that purpose, it is far clearer if we use additive notation (that is, with $\star$ as $+$, $\eta$ as $0$) and work with the reverse ordering.  To be precise we will actually consider structures $\langle V,+,0,\leq\rangle$ where $\langle V,+,0,\geq\rangle$ is a strongly Flagg quantale.  In a slight abuse of notation, we will still refer to these structures as Flagg (or strongly Flagg, etc.) quantales, although now the characteristic properties are precisely the duals of those in Section~\ref{subsecquantprelims}.

\begin{definition}\label{defvmetr}
Let $\alglatt$ be a Flagg quantale.  A {\it $\V$-metric space} consists of a pair $(M,d)$, $d:M\times M\to\V$, where $d$ satisfies, for all $x,y,z\in M$,
	\begin{enumerate}
		\item (Reflexivity) $d(x,x)=0$
		\item (Symmetry) $d(x,y)=d(y,x)$
		\item (Triangle Inequality) $d(x,y)\le d(x,z)+ d(z,y)$
		\item (Equality) if $d(x,y)=0$, $x=y$.
	\end{enumerate}
	If $(M,d)$ satisfies only (1), (2), and (3), we say that it is a {\it $\V$-pseudometric space}.
\end{definition}

\begin{example}\label{ExVSpaces}
\begin{enumerate}
\item Let $\V=\langle V,+,0,\ge \rangle$ be Flagg.  Then $V$ itself is a $\V$-space, with $d(x,y)=(y\dotminus x)+(x\dotminus y)$.  Recall that $\dotminus$, truncated subtraction, is the left adjoint to $+$ (Remark~\ref{rmkadj}).
\item\label{truthmetr} Consider the (dual of the) quantale of truth values, $\Omega_\wedge$, introduced in Example~\ref{exflaggquant}(1).  An $\Omega_\wedge$-pseudometric space will consist simply of a set $M$ equipped with an equivalence relation $\equiv_{M}$ given by $x\equiv_{M} y$ if and only if $d_M(x,y)=0$.  In an $\Omega_{\wedge}$-metric space, this equivalence relation is trivial: $\Omega_\wedge$-metric spaces are simply discrete sets.
\item\label{distmetr} Consider $[0,\infty]_+$, the (dual of the) quantale of distances given in Example~\ref{exflaggquant}(2). An $R^+$-pseudometric space $\langle M,d\rangle$ yields a distance mapping $d:M\times M\to [0,\infty]$. If $d$ is reflexive, transitive, symmetric and separated---that is, it satisfies the condition \ref{defvmetr}(4)---then $\langle M,d\rangle$ is a generalized pseudometric space.  The $[0,\infty]_+$-metric spaces, moreover, are precisely the generalized metric spaces.
\item\label{distdistmetr} Consider $\Delta_\oplus$, the (dual of the) quantale of distance distribution functions with operation induced by \L ukasiewicz addition, described in Example~\ref{exflaggquant}(5).  A $\Delta_\oplus$-pseudometric space consists of a set $M$, and a function $p:M\times M\to\Delta_\oplus$---assigning to each pair of elements of $M$ not a distance, but rather a distribution of distances---satisfying the following conditions:
	\begin{itemize}
	\item $p(x,x)=\eta$
	\item $p(x,y)=p(y,x)$
	\item $p(x,y)\leq p(x,z)\oplus p(z,y)$
	\end{itemize}
	which are, of course, precisely the axioms of a {\it probabilistic pseudometric space}.  If we consider $\Delta_\oplus$-metric spaces, i.e. those with the added condition that $p(x,y)=\eta$ implies $x=y$, we obtain a {\it probabilistic metric space} (see \cite{probmet}).
\item\label{framemetr} In (the dual of) a continuous frame $\Omega$, made into a quantale with $+$ as $\vee$, we note that the triangle inequality becomes
$$d(x,y)\leq d(x,z)\vee d(z,y),$$
so any $\Omega$-(pseudo)metric space is necessarily ultrametric.
\end{enumerate}
\end{example}

\begin{remark}\label{rmkvsps}
\begin{enumerate}
\item We note that the same notion can be packaged in terms of enriched categories---as noted in \cite{lawveremetr} for the quantale $[0,\infty]_{+}$ of extended nonnegative reals, there is an equivalence of categories between $[0,\infty]_{+}$-metric spaces (that is, generalized metric spaces) and categories enriched over $[0,\infty]_{+}$, with the latter thought of as a {\it symmetric, monoidal closed category}, $+$ as the tensor product).  In fact, \cite{lawveremetr} shows, in essence, that this holds for an arbitrary (nice) quantale, and considerable work has gone into developing this perspective, e.g. \cite{hofreisconv}, \cite{hofwaszenrich}.  We will not stress this point of view, but many of the examples of \cite{lawveremetr} loom large here, alongside those of \cite{Fl} and \cite{flaggkoppcontsp}.

\item There is, moreover, a close correspondence between $\Omega$-(pseudo)metric spaces, (separated) presheaves on $\Omega$, and $\Omega$-fuzzy sets---this is clear via their mutual connection to $\Omega$-sets, cf. \cite{fourmanscott} and \cite{fuzzysetsii}.  We note that the reflexivity axiom is a brutal restriction: it rules out all but those presheaves whose local sections are necessarily global, and fuzzy sets that are essentially crisp.  To obtain interesting results along these lines, one must drop reflexivity and shift to {\it partial} $\Omega$-metric spaces.  We return to this idea in Section~\ref{secpartial}.\end{enumerate}
\end{remark}

We note, too, that a $\V$-metric space comes equipped with a natural topology, allowing us to make sense of completeness in this context:

\begin{definition}
Let $(M,d)$ be a $\V$-(pseudo)metric space, $\V$ a Flagg quantale. 
\begin{enumerate}
\item Given $x\in M$ and $\E \gg 0$, we define the {\it open ball of radius $\E$ and centered at $x$} as
\[ B_\E(x)=\{y\in M\mid \E\gg d(x,y)\}. \]
These open balls form the base for a topology on $M$ (closure under finite intersections follows from Remark~\ref{rmkflaggterm}).
\item We say that a sequence $(x_n)_{n\in\omega}$ in $M$ is {\it Cauchy} if for every $\E\gg 0$, there is $N\in\omega$ such that $\E\gg d(x_n,x_m)$ for all $n,m\ge N$.
\item We say that a sequence $(x_n)_{n\in\omega}$ converges to $x\in M$ if $\E\gg 0$ there is an $N\in\omega$ such that $x_n\in B_\E(x)$ for all $n\ge N$.
\item We say that $(M,d)$ is {\it Cauchy complete} (or, if there is no risk of confusion, {\em complete}) if every Cauchy sequence converges.
\end{enumerate}
\end{definition}

Although these notions are well-defined in a general Flagg quantale, we note that the additional assumption of sequential approximation, Definition~\ref{defsflagg}(1), has many clear benefits:
\begin{enumerate}
	\item It ensures that sequential closure corresponds to topological closure, even in this more general framework.
	\item It is precisely what is needed to ensure that Cauchy completeness corresponds to Lawvere's category-theoretic notion (cf.  \cite[2.27]{hofreisconv}), which corresponds to, say, sheaf-theoretic completeness.
	\item Building on (2) above, sequential approximation is sufficient to guarantee the existence of completions of $\V$-metric spaces, which we will need in connection with $\V$-AECs, beginning in the next section.
\end{enumerate}

We restrict our attention to complete $\V$-pseudometric and $\V$-metric spaces, $\V$ a Flagg quantale with sequential approximation.  We have many choices when it comes to the proper notion of morphism between such spaces, hence several options in defining the associated categories.  We make the following definitions:

\begin{definition}\label{defmetcats} We denote by $\metr{\V}$ the category of complete $\V$-metric spaces with morphisms the isometries, and by $\metrc{\V}$the category of complete $\V$-metric spaces with morphisms the nonexpanding (that is, {\it $1$-Lipshitz continuous}) maps.
We denote by $\psmetr{\V}$ and $\psmetrc{\V}$ the corresponding categories of complete $\V$-pseudometric spaces.\end{definition}

\begin{remark}\label{contractrmk} As is often the case in this paper, weaker choices of morphism are possible.  The choice of nonexpanding maps is motivated by the connection to enriched category theory (it is shown in \cite{lawveremetr}, for example, that $\metrc{{[0,\infty]_+}}$ is the enriched category $[0,\infty]_+$-{\bf Cat}) and the fact that, in the case $\Omega$ is a locally compact locale, $\metrc{\Omega}$ is equivalent to $\shcat{\Omega}$, the category of sheaves on $\Omega$ (this is the gist of \cite[pp. 1152-1153]{fuzzysetsi})
\end{remark}

\begin{proposition}\label{catsfinpres}
For any Flagg quantale with sequential approximation $\V$, the categories $\metr{\V}$, $\metrc{\V}$, $\psmetr{\V}$, and $\psmetrc{\V}$ are $\aleph_1$-accessible with directed colimits, and in each case the internal size of an object $M$ (see Example~\ref{accexams}(2)(c)) is $dc(M)$; that is, the density character of the space $M$.
\end{proposition}

This is easy to verify, with the argument paralleling the one for complete metric spaces in \cite[3.5]{lrmcaec}.

\section{$\V$-Abstract Elementary Classes}\label{secvaecs}

We now, finally, begin to introduce logic into the picture, beginning with the description of structures with underlying $\V$-pseudometric (or $\V$-metric) spaces.

\begin{definition}\label{defmetrstruct}
Let $\V$ be a Flagg quantale with sequential approximation.  Given a finitary signature $L$, a {\it $\V$-pseudometric structure} for $L$ is a tuple $\C{M}=\langle (M,d_M); \sigma^{\C{M}} \rangle_{\sigma\in L}$, with $d_M$ a $\V$-pseudometric, and interpretations of $L$-symbols satisfying the following conditions:
\begin{enumerate}
\item If $c\in L$ is a constant symbol, $c^{\C{M}}\in M$.
\item If $R\in L$ is a relational symbol of arity $1\le k<\omega$, then $R^{\C{M}}: M^k\to V$ is a nonexpanding map. 
\item If $f\in L$ is a function symbol of arity $1\le k<\omega$, $f^{\C{M}}: M^k\to M$ is a nonexpanding map.
\end{enumerate}
A {\it $\V$-metric structure} for $L$ is defined in precisely the same way, but with the requirement that $(M,d)$ be a $\V$-metric space. We say that a $\V$-pseudometric structure (respectively, $\V$-metric structure) is {\it complete} if $d_M$ is a Cauchy complete $\V$-pseudometric (respectively, $\V$-metric).  
The notion of {\it substructure}, $\subseteq_{\V,L}$, is clear: a sub-$\V$-metric structure is a subspace where the interpretations of symbols from $L$ are those inherited from the larger structure.
\end{definition}

\begin{note}
	We embrace the usual abuse of notation, conflating a $\V$-metric structure $\C{M}$ with its underlying space $\langle M,d\rangle$ or, indeed, with $M$ itself.
\end{note}

\begin{remark}
	As will become clear in Remark~\ref{vmetrstrexs} below, we lose some generality by insisting that the interpretation of function and relation symbols be nonexpanding, rather than uniformly continuous, as is more common.  In particular, Banach spaces will not fit into the current framework.  We feel comfortable with this choice for the following reasons:
	\begin{itemize}
		\item Nothing about the arguments presented here depends on the nature of this choice---those who wish to replace nonexpanding maps with continuous or uniformly continuous maps may do so, with no change in the results or proofs on offer.
		\item The connection between generalized metric structures and sheaf structures, hinted at in Section~\ref{secpartial}, involves an equivalence of categories in which sheaf operations correspond precisely to nonexpanding maps.  This suggests our choice is a natural one.
	\end{itemize}
\end{remark}

We choose a very strong notion of morphism between $\V$-metric structures, in keeping with the notion familiar from the abstract model theory of continuous logic.  

\begin{definition}\label{embeddings}
Let $M_1$ and $M_2$ be $\V$-pseudometric structures in signature $L$. A {\it $(\V,L)$-embedding} of $M_1$ into $M_2$ is a mapping $h:M_1\to M_2$ such that:
\begin{enumerate}
\item $h$ is 
an (injective) isometry; that is, for any $x,y\in M_1$, 
$$d_2(h(x),h(y))= d_1(x,y).$$
\item for any constant symbol $c\in L$, 
	$$h(c^{M_1})=c^{M_2}.$$ 
\item for any relation symbol $R\in L$ of arity $k$, and $x_1,\cdots,x_k\in M_1$, 
	$$R^{M_1}(x_1,\cdots,x_k)=R^{M_2}(h(x_1),\cdots,h(x_k)).$$
\item for any function symbol $f\in L$ of arity $n$ and $x_1,\cdots, x_n\in M_1$, $$h(f^{M_1}(x_1,\cdots.x_n))=f^{M_2}(h(x_1),\cdots,h(x_n)).$$
\end{enumerate}
\end{definition}

\begin{remark}
	We must explicitly require that our isometries are injective in the pseudometric case---without the assumption that distinct elements are at nonzero distance from one another, isometries need not be injective.
\end{remark}

In what follows, we will be interested primarily in \emph{complete} $\V$-pseudometric or $\V$-metric structures---this is motivated by the fact that in continuous model theory, and in mAECs, completeness is typically assumed.  To that end, we denote by $\psmetrstr{\V}{L}$ the category of all complete $\V$-pseudometric structures over $L$ and all $(\V,L)$-embeddings between them.  We define $\metrstr{\V}{L}$ analogously.

\begin{example}\label{vmetrstrexs}
\begin{enumerate}
    \item Following Example~\ref{ExVSpaces}(\ref{truthmetr}), an object of $\psmetrstr{{\Omega_{\wedge}}}{L}$ will consist of an underlying set $M$, equipped with an equivalence relation $\equiv_M$ induced by the metric, with interpretions of the function and relation symbols of $L$ well-defined with respect to $\equiv_M$.  The category $\metrstr{{\Omega_{\wedge}}}{L}$, moreover, will be precisely the category of (discrete) $L$-structures, in the usual sense.
	\item With the quantale of distances, $[0,\infty]_+$, things are very familiar, indeed.  The objects of $\metrstr{{[0,\infty]_+}}{L}$ have underlying complete generalized metric spaces.  Operations and relations are taken to be nonexpanding, hence also {\it uniformly continuous} (as is more often assumed, cf. \cite{Za11} or \cite{hirhyt}).  The morphisms are simply the symbol-preserving isometries.
	\item The category $\metrstr{{{\Delta_\oplus}}}{L}$ is very similar in its description, with the proviso, again, that distances are replaced by probability distributions thereof.
\end{enumerate}	
\end{example}

We now define the central notion: $\V$-AECs.  While we require completeness of the underlying spaces, we note, again, that this is intended to bring our notion in line with existing ones: none of the results in this paper, including Theorem~\ref{VTameness}, actually make use of completeness.

\begin{definition}\label{defvaecs}
Let $\V$ be strongly Flagg with sequential approximation.  Let $\C{K}$ be a class of complete $\V$-metric structures for a signature $L$ and $\ordencea$ a binary relation defined on $\C{K}$. We say that $\AEC[K]{}$ is a {\em $\V$-abstract elementary class} (or {\em $\V$-AEC}) iff:
\begin{enumerate}
\item $\ordencea$ partially orders $\C{K}$ and $M_1\ordencea M_2$ implies $M_1\subseteq M_2$.
\item $\C{K}$ respects $L$-isomorphisms; that is, if $h:M_1\stackrel{\approx}{\to} M_2$ is an $L$-isomorphism then $M_1\in \C{K}$ implies $M_2\in \C{K}$, and if $M\ordencea M_1$ then $f[M]\ordencea M_2$.
\item (Coherence) if $M_0\subseteq M_1\ordencea M_2$ and $M_0\ordencea M_2$ then $M_0\ordencea M_1$.
\item $\C{K}$ is closed under directed colimits.
\item (Downward L\"owenheim-Skolem) There exists an infinite cardinal $LS(\C{K})$ such that for any $M\in \C{K}$ and any $A\subseteq M$ there exists a $N\in\C{K}$ with $\dc{N}\le (LS(\C{K})+|A|)^+$ such that $A\subseteq N$ and $N\ordencea M$.
\end{enumerate} 
We define a {\em pseudo-$\V$-AEC} similarly, with $\C{K}$ a class of complete $\V$-pseudometric structures.
\end{definition}

\begin{remark}\label{rmkvaecsexs}
Notice that $\Omega_\wedge$-AECs are simply AECs in the usual sense (see Example~\ref{vmetrstrexs}(1)).  One might expect that an $[0,\infty]_+$-AEC would be an mAEC, and vice versa.  In fact, they differ in two important respects: 
\begin{enumerate}\item The underlying spaces of a $[0,\infty]_+$-AEC, of course, have metrics that take values in $[0,\infty]$, rather than $[0,\infty)$.
\item To reiterate a point from Example~\ref{vmetrstrexs}(2), the  interpretations of function and relation symbols in an $[0,\infty]_+$-AEC, or $\V$-AEC more generally, are taken to be nonexpanding.  In an mAEC, they are typically assumed only to be continuous, or uniformly continuous.  As we have mentioned, there is considerable freedom here: the story that follows will not change one iota if we opt for one of these weaker options.\end{enumerate}
\end{remark}

\begin{definition}[$\ordencea$-embedding]
Let $\C{K}$ be a (pseudo-)$\V$-AEC. A {\it $\ordencea$-embedding} (or {\it $\C{K}$-embedding}) is a function $f:M\to N$, $M,N\in \C{K}$, that is an $L(\C{K})$-embedding in the sense of Definition~\ref{embeddings} and such that $f[M]\ordencea N$.
\end{definition}

\begin{proposition}\label{propvaecsize}
	For any $\V$-AEC (or pseudo-$\V$-AEC) $\ck$ in a finitary signature $L$, $\ck$ is an ${\lsk}^+$-accessible category with arbitrary directed colimits, and, moreover, these colimits are concrete.  Furthermore, an object $M\in\ck$ is $\kappa$-presentable for $\kappa>\lsk$ if and only if $\dc{M}<\kappa$.  That is, the internal size of an object is precisely its density character.
\end{proposition}
\begin{proof} The proof runs along the same lines as the corresponding proof for mAECs, \cite[3.1]{lrmcaec}.
\end{proof}

Once again, this argument goes through if we take the interpretations of function and relation symbols to be continuous, or uniformly continuous.

\section{Galois types in $\mathbb{V}$-AECs.}

We now define the notion of type that we will use: the \emph{Galois types} (or \emph{orbital types}) mentioned in the introduction.  Up to a very minor translation, our characterisation is the same as the one used across the literature on AECs and mAECs: types are identified with equivalence classes of pointed extensions of a common model.

\begin{definition}\label{Pairs}
Let $\V$ be a strongly Flagg quantale with sequential approximation, and let $\C{K}$ be a pseudo-$\V$-AEC.  Given $M\in \C{K}$, we define $\mathbb{K}^2_{M}$, the collection of pointed extensions of $M$, as follows: $$\mathbb{K}^2_{M}=\{(f,a)\mid f:M\stackrel{\ordencea}{\to} N, a\in N\}.$$
\end{definition}

\begin{definition}\label{EquivPairs}
Let $\V$ be a strongly Flagg quantale with sequential approximation, and let $\C{K}$ be a pseudo-$\V$-AEC.  Given $M\in\ck$ and a pair 
$(f_i:M\to N_i,a_i)\in \mathbb{K}^2_{M}$, $i=0,1$, 
we say that $(f_0,a_0)$ and $(f_1,a_1)$ are {\it equivalent}, denoted 
$$(f_0,a_0)\sim_{M} (f_1,a_1),$$ 
if and only if there exists an $N\in \mathcal{K}$ and $\ordencea$-embeddings $g_i:N_i\to N$, $i=0,1$, such that $g_0\circ f_0=g_1\circ f_1$ and $g_0(a_0)=g_1(a_1)$.
\end{definition}

Notice that $\sim_M$ is not necessarily an equivalence relation, but certainly will be with the following additional assumption:

\begin{definition}\label{defAP}
We say that a class $\ck$ has the \emph{amalgamation property} (or \emph{AP}) if for any $M\in\ck$ and pair of $\ck$-morphisms $f_i:M\to M_i$, $i=0,1$, there is an $N\in\ck$ and pair of maps $g_i:M_i\to N$ such that $g_0f_0=g_1f_1$.
\end{definition}

We will assume in the sequel that all of our classes satisfy both AP and the \emph{joint embedding property} (or \emph{JEP}): any two objects embed into a third.  The latter assumption is needed to guarantee the existence of a monster model, Remark~\ref{rmkmonstmetr} below.

\begin{definition}\label{SetTypes} Let $\V$ be a strongly Flagg quantale with sequential approximation, and let $\C{K}$ be a pseudo-$\V$-AEC.  For any $(f,a)\in \mathbb{K}^2_{M}$, we define the Galois type of $(f,a)$ over $M$ to be
$$\gatp(a/f)=[(f,a)]_{\sim_{M}}$$ 
where $[(f,a)]_{\sim_{M}}$ is the class of all pointed extensions in $\mathbb{K}^2_{M}$ $\sim_{M}$-equivalent to $(f,a)$.  We define the set of Galois types over $M\in\ck$ to be
$$\gaS(M)=\{\gatp(a/f)\mid (f,a)\in\mathbb{K}^2_{M} \},$$
Moreover, 
We say that a pair $(f,a)\in\mathbb{K}^2_{M}$ \emph{realizes} a Galois type $p$ over $M$, and write $(f,a)\models p$, just in case $[(f,a)]_{\sim_M}=p$.\end{definition}

Notice that we consider only Galois types of elements; that is, Galois $1$-types.  Nonetheless, all of the results here extend easily to types of finite tuples.  We now show that the set of Galois types comes with a natural $\V$-metric structure. 

\begin{definition}
Let $\V$ be a strongly Flagg quantale with sequential approximation, and let $\C{K}$ be a pseudo-$\V$-AEC.  Given $M$ in $\ck$ and $p=\gatp(a_0/f_0),q=\gatp(a_1/f_1)\in \gaS(M)$, we define $\dist(p,q)\in V$ as follows:
\[\dist(p,q)=\bigwedge\{ d_N(g_0(a_0),g_1(a_1))\mid g_i:N_i\to N, g_0\circ f_0 = g_1\circ f_1  \}\].
\end{definition}

\begin{remark}\label{rmkmonstmetr}
As mentioned in the introduction, it is more usual (see e.g. \cite{hirhyt}, \cite{Za11}, \cite{boneyzam}) to invoke a large, highly saturated {\it monster model} $\mathbb M$ which contains all relevant structures, to identify types over $M\in\ck$ with orbits in $\mathbb M$ under automorphisms fixing $M$, and to compute the distance between types $p_i=\gatp(a_i/f_1)$ over $M$, $i=0,1$, as the Hausdorff distance between the corresponding orbits:
$$\dist(p_0,p_1)=\inf \{d_{\mathbb M}(f(a_0),g(a_1))\mid f,g\in \text{Aut}(\mathbb{M}/M)\}$$
As mentioned after Definition~\ref{defAP}, AP and JEP---which we assume---are sufficient to guarantee the existence of such a model.  This perspective dramatically simplifies computational matters, and precisely matches the distance notion defined above.  From this perspective, we may speak of the realization of a type $p$ by an element $a\in\mathbb M$, which we denote $a\models p$, meaning, in particular, that $(i_M,a)\models p$, where $i_M:M\to \mathbb M$ is the inclusion.\end{remark}

The following property is reminiscent of \cite[3.2]{hirhyt}, insofar as it shows that proximity of types is witnessed by their realizations.  We include the proof here for the sake of completeness, and because we must now work with the totally above relation, $\gg$, rather than simple inequality.

\begin{proposition}\label{inf}
Let $\V$ be a strongly Flagg quantale, $\ck$ a pseudo-$\V$-AEC.  For any $\varepsilon\gg 0$ in $V$, for $p,q\in \gaS(M)$, $M\in\ck$, and any realization $a\models p$, there exists a realization $b\models q$
such that $\dist(p,q)+\varepsilon\gg d_\monster(a,b)$.
\end{proposition}
\begin{proof} Let $\varepsilon \gg 0$ in $V$ and $a\models p$. By monotonicity, Definition~\ref{defsflagg}(1), $\dist(p,q)+\E\gg \dist(p,q)=\bigwedge \{d_\monster(a,b): a\models p, b\models q\}$. By constructively complete distributivity, there exist $a'\models p$ and $b'\models q$ such that $\dist(p,q)+\E\gg d_\monster(a',b')$. Since $a,a'\models p$, there exists $f\in Aut(\mathbb{M}/M)$ such that $f(a') = a$. Notice that
$$\dist(p,q)+\varepsilon\gg d_\monster(a',b')=d_\monster(f(a'),f(b'))=d_\monster(a,f(b')).$$
We have that $f(b')\models q$, so $b=f(b')$ is the required element.
\end{proof}
\vspace{1mm}

Following \cite{Za11}, we have:

\begin{proposition}\label{dtypes_drealizations}
Let $\V$ be a strongly Flagg quantale, $\ck$ a pseudo-$\V$-AEC.  For all $\E\gg 0$ and $p,q\in \gaS(M)$, $M\in\ck$, such that $\E\gg \dist(p,q)$ and $b\models q$, then there exists $a_{\varepsilon}\models p$ such that $\varepsilon\gg d_\monster(a_{\varepsilon},b)$. 
\end{proposition}
\begin{proof}
Since $\E\gg \dist(p,q)=\bigwedge\{d(a,b): a\models p, b\models q \}$ , there exist $a'\models p$ and $b'\models q$ such that $\E\gg d(a',b')$, by constructively complete distributivity. Since $b',b\models q$, there exists $f\in Aut(\monster/M)$ such that $f(b')=b$. Notice that $f(a')\models p$ (because $a'\models p$ and $f$ fixes $M$ pointwise). Since $f$ is an isometry, $d(f(a'),b)=d(f(a'),f(b'))= d(a',b')$, meaning that $\E\gg d(f(a'),b)$. Notice that the statement of the proposition holds for $a_\E=f(a')$.
\end{proof}

\begin{proposition}\label{GaTypesVSpace}
Let $\V$ be a strongly Flagg quantale.  For any $M$ in a pseudo-$\V$-AEC $\ck$, $(\gaS(M), \dist)$ is a $\V$-pseudometric space.
\end{proposition}
\begin{proof}
\begin{enumerate}
\item (Reflexivity) It is clear that for any $p\in \gaS(M)$, $\dist(p,p)=\bigwedge\{d(a,b):a\models p, b\models p\}=0$.
\item (Symmetry) Let $p,q\in \gaS(M)$. Since $\monster$ is a $\V$-pseudometric space, $d$ is symmetric, and thus
$$\dist(p,q)=\bigwedge\{d(a,b): a\models p, b\models q\}=\bigwedge\{d(b,a): a\models p, b\models q\}=\dist(q,p).$$
\item (Triangle Inequality) Let $\E\gg 0$. By Lemma~\ref{epsilonovertwo}, there exists $\delta\gg 0$ such that $\E\gg 2\delta\gg 0$. By Proposition~\ref{inf} and the fact that $\gg$ refines $\geq$, for any $a\models p$ there exists $b_\delta\models q$ such that $d(a,b_\delta)\le \dist(p,q)+\delta$. For this $b_\delta$, the same reasoning yields $c_\delta\models r$ such that $d(b_\delta,c_\delta)\le \dist(q,r)+\delta$. So
\begin{align*}(\dist(p,q)+\dist(q,r))+\E & \gg (\dist(p,q)+\dist(q,r))+(2\delta)\\ 
 & \geq d(a,b_\delta)+d(b_\delta,c_\delta)\\
 & \geq d(a,c_\delta).\end{align*} 
By definition of $\dist(p,r)$, $\dist(p,r)\le d(a,c_\E)\le (\dist(p,q)+\dist(q,r))+\E$. Since $\E\gg 0$ is arbitrary, Lemma~\ref{Closeness} implies that $\dist(p,r)\le \dist(p,q)+\dist(q,r)$.
\end{enumerate}
\end{proof}

\begin{remark}
Notice that Flaggness of $\V$ (in the form of the ``$\E/2$'' property) features only in the proof of the triangle inequality for $\dist$, and sequential approximation is not needed.	
\end{remark}

Even if we are working in a $\V$-AEC, $(\gaS(M), \dist)$ need not be a $\V$-metric space.  We require a further assumption:

\begin{definition}\label{CTP}
We say that a $\V$-AEC satisfies the \emph{continuity of types property} (or \emph{CTP}) if it has the following property: for any $a\in \monster$ and $p\in \gaS(M)$, if for any $\E\gg 0$ there exists some $b\models p$ such that $d(a,b)\le \E$, then $a\models p$.
\end{definition}

The following is an easy exercise (see e.g. the discussion following \cite[3.2]{hirhyt}):

\begin{proposition}\label{propctpmetr}
	For any $M$ in a $\V$-AEC $\ck$ satisfying the CTP, $(\gaS(M), \dist)$ is a $\V$-metric space.
\end{proposition}

\section{Tameness in $\mathbb{V}$-AECs.}\label{sectame}
In this section, we prove the set-theoretical consistency of a metric version of tameness in pseudo-$\mathbb{V}$-AECs, by a slight generalization of the methods in~\cite{LiRo17}. The proof given in~\cite{LiRo17} makes essential use of the property of $\mathbb{R}$ that given any real  number $\delta>0$ we can find a sequence $(\delta_n)_{n<\omega}$ ($\delta_n>0$) which converges to $\delta$. In the argument that follows, we attain the same effect through the sequential approximation property of strongly Flagg quantales (Definition~\ref{defsflagg}(2)).

In what follows, we will speak freely of \emph{presentability}, in the sense of Definition~\ref{defacc}(2).  Recall that, by Proposition~\ref{propvaecsize}, an object $M$ in a (pseudo)-$\V$-AEC $\ck$ is $\kappa$-presentable for $\kappa>\lsk$ if and only if $\dc{M}<\kappa$.

\begin{definition}\label{deftame}
	Let $\V$ be strongly Flagg.  We say that a pseudo-$\V$-AEC $\C{K}$ is {\it $\kappa$-$\V$-tame} if for any $\epsilon\gg 0$ there is a $\delta\gg 0$ such that for any $M\in \C{K}$ and Galois types $p_i=\gatp(a_i/f_i)$ with $f_i:M\to N_i$, $i=0,1$, if $\delta\gg\dist(p_0\circ \chi,p_1\circ \chi)$ for all $\chi:X\to M$ with $X\in\ck$ $\kappa$-presentable, then $\epsilon\gg\dist(p_0,p_1)$.
	
We say that $\C{K}$ is {\it strongly $\kappa$-$\V$-tame} if the above holds with $\delta=\epsilon$.

We say that $\C{K}$ is {\it (strongly) $\V$-tame} if it is (strongly) $\kappa$-$\V$-tame for some $\kappa$.
\end{definition}

In short, $\ck$ is strongly $\kappa$-$\V$-tame if $\E$-closeness of types over a general $M\in\ck$ is determined entirely by $\E$-closeness of their restrictions to subobjects $X\to M$ of size less than $\kappa^+$.  Recalling that accessible categories are precisely those determined by a set of small objects, this suggests the outline of the proof of tameness from almost strongly compact cardinals---if we can show that a suitable category of $\E$-close pairs of types is accessible, tameness should follow.  The broad outline will be familiar to readers of \cite{LRclass} and \cite{LiRo17}.  We note that the argument is nearly identical to that given for the analogous result on tameness of mAECs, \cite[5.5]{LiRo17}.  As such, we elide certain details, while carefully accounting for the minor modifications required in our context.

We now turn to the details.  We begin by defining categories $\C{L}_\E$ and $\C{L}$, together with a forgetful functor $G_\E:\C{L}_\E\to\C{L}$, that capture $\E$-closeness of types.

\begin{definition} \label{CategoryLE}
Let $\V$ be a strongly Flagg quantale, $\ck$ a pseudo-$\V$-AEC, and $U:\ck\to\psmetr{\V}$ the forgetful functor.  Given any $\E\gg 0$, as in  \cite{LiRo17} we define the category $\mathcal{L}_\E$ as follows:

\begin{enumerate}
\item The objects are of the form $(f_0, f_1, g_0, g_1, a_0, a_1)$ where $f_i: M \rightarrow N_i$, $g_i: N_i \rightarrow N$  are morphisms in $\mathcal{K}$ and $a_i \in UN_i$, $i=0,1$, such that there exists an isometry $h: 2_{\epsilon} \rightarrow UN$ (where $2_{\epsilon}$ is the 2-pointed $\V$-space of diameter $\epsilon$) such that $h(0) = U(g_0)(a_0)$ and $h(1) = U(g_1)(a_1)$.

$$\xymatrix@=10pt{ 
 & & &  2_{\epsilon} \ar[ld]^{h}\\
UN_0\ar[rr]^{U(g_0)} &  & UN &  & \\
N_0 \ar[r]^{g_0} & N  \\ 
M \ar[u]^{f_0} \ar[r]_{f_1} & N_1 \ar[u]_{g_1} 
& UN_1\ar[uu]_{U(g_1)} & 
}
$$

\item Let $(f_0, f_1, g_0, g_1, a_0, a_1), (f'_0, f'_1, g'_0, g'_1, a'_0, a'_1) \in Obj(\mathcal{L}_{\epsilon})$. A morphism from $(f'_0, f'_1, g'_0, g'_0, a'_0, a'_1)$ to $(f_0, f_1, g_0, g_0, a_0, a_1)$ in $\mathcal{L}_\E$ corresponds to a tuple $(t, u, v, w)$, where $t: N' \rightarrow N$, $u: N'_0 \rightarrow N_0$,  $v: M \rightarrow M$, $w: N'_1 \rightarrow N_1$ are $\mathcal{K}$-morphisms such that the following diagram commutes:

$$\xymatrix@=12pt{
 & N'\ar@/^{-3mm}/[ddd]_{t} & \\
 N_0'\ar[d]_{u}\ar@/^{2mm}/[ru]^{g'_0} & M\ar[l]_{f'_0} \ar[r]^{f'_1}\ar[d]^{v} & N_1'\ar[d]_{w}\ar@/_{2mm}/[lu]_{g'_1}\\
 N_0\ar@/_{2mm}/[rd]_{g_0} & M\ar[l]_{f_0} \ar[r]^{f_1} & N_1\ar@/^{2mm}/[ld]^{g_1}\\
 & N & 
}$$
with the added conditions that $U(w)(a_i')=a_i$, $i=0,1$.  It follows that $U(t)h'=h$.
\end{enumerate}
\end{definition}

\begin{definition}\label{CategoryL}
As in  \cite{LiRo17}, we define the category $\mathcal{L}$ as follows:
\begin{enumerate}
\item $Obj(\mathcal{L})$: $(f_0, f_1, a_0, a_1)$, where $f_0:M \rightarrow N_0$, $f_1:M \rightarrow N_1$ are $\C{K}$-morphisms and $a_0\in U(N_0),a_1\in U(N_1)$.

$$\xymatrix@=12pt{N_0&\\M\ar[u]^{f_0}\ar[r]_{f_1}&N_1}$$

\item  Given $(f_0, f_1, a_0, a_1),(f_0, f_1, a_0, a_1) \in Obj(\mathcal{L})$, a morphism from $(f'_0, f'_1, a'_0, a'_1)$ to $(f_0, f_1, a_0, a_1)$ in $\mathcal{L}$ is a tuple of the form $(u, v, w)$, where $u: N'_0 \rightarrow N_0$,  $v: M \rightarrow M$, $w: N'_1 \rightarrow N_1$ are $\mathcal{K}$-morphisms such that the following diagram commutes:

$$\xymatrix@=12pt{ 
 N'_0\ar[d]_{u} &   M \ar[l]_{f'_0} \ar[r]^{f'_1} \ar[d]_{v}& N'_1\ar[d]_{w}    \\
 N_0  & M \ar[l]^{f_0} \ar[r]_{f_1} & N_1     \\
}
$$
with $U(w)(a_i')=a_i$, $i=0,1$.
\end{enumerate}
\end{definition}

\begin{definition}\label{forgfunct} We denote by $G_\E$ the forgetful functor from $\cl_\E$ to $\cl$, 
$$ G_\E : (f_0,f_1,g_0,g_1,a_0,a_1)\mapsto(f_0,f_1,a_0,a_1).$$\end{definition}

\begin{remark}\label{thmprep}
	\begin{enumerate}
		\item Notice that if $(f_0,f_1,a_0,a_1)$ belongs to the full image of $G_\E$, then the distance between the corresponding Galois types $\gatp(a_0/f_0)$ and $\gatp(a_1/f_1)$ is at most $\E$.
		\item The full image of $G_\E$ is closed under precomposition, by the same argument as in the proof of \cite[5.5]{LiRo17}.  In particular, the full image is powerful.
		\item One can easily verify that $\C{L}$, $\C{L_\E}$, and $G_\E$ are accessible.  Indeed, there exists a cardinal $\lambda$ such that for all $\epsilon\gg 0$, $G_\E$ is $\lambda$-accessible and preserves $\lambda$-presentable objects, by the remark following \cite[2.19]{adros}.
	\end{enumerate} 
\end{remark}

\begin{remark}\label{rmkrevmul}In the following theorem, we invoke the cardinal parameter $\mu_\cl$ associated with $\cl$, which was mentioned in the paragraph following Theorem~\ref{BTRTh}.  In fact, we will require a minor modification of the standard definition of that parameter, taking a sufficiently high seed cardinal $\gamma_\cl$ (in the terminology of \cite{BT-R}) to ensure that $\mu_\cl>2^{|\V|}$.  We denote the result by $\mu_{\cl}^{\V}$.\end{remark}

\begin{theorem}\label{VTameness}
Assuming the existence of a $L_{\mu_{\cl}^{\V},\omega}$-compact cardinal,  every pseudo-$\V$-AEC is strongly $\V$-tame.
\end{theorem}
\begin{proof}
Let $\C{K}$ be a pseudo-$\V$-AEC with forgetful functor $U:\ck\to\psmetr{\V}$, and let $\C{L}$, $\C{L_\epsilon}$, and $G_\E: \C{L}_\E\to \C{L}$ be as described in~\ref{CategoryLE}, \ref{CategoryL}, and \ref{forgfunct}.  Fix $\lambda$ a regular cardinal such that for any $\E\gg 0$ we have that $G_\E$ is $\lambda$-accessible and preserves $\lambda$-presentable objects (see Remark~\ref{thmprep}(3)).

Let $\kappa$ be a $L_{\mu_{\cl}^{\V},\omega}$-compact cardinal. Since $\mu_{\cl}^{\V}\triangleright\lambda$, $G_\E$ is $\mu_{\cl}^{\V}$-accessible and preserves $\mu_{\cl}^{\V}$-presentable objects.  So, by Theorem~\ref{BTRTh}, we have that the powerful image (that is, the full image, Remark~\ref{thmprep}(2)) of any $G_\E$ is $\kappa$-accessible.

Let $\E\gg 0$ and $p_i=\gatp(f_i:M\to N_i,a_i)$, $i=0,1$, such that for all $\kappa$-presentable subobjects $h:X\to M$ we have that $\E\gg d(p_0\restr h,p_1\restr h)$, where $p_i\restr h$ denotes the Galois type $\gatp(f_i\circ h: X\to N_i,a_i)$. By precisely the argument of \cite[5.5]{LiRo17}, there exists a cofinal set $D_\delta$ of morphisms $\chi: X\to M$ ($X$ $\kappa$-presentable) such that $d(p_0\restr \chi,p_1\restr \chi)=\delta$ for a single $\delta$, $\E\gg\delta$.  That argument, which proceeds by contradiction and hinges on counting potential witnesses to the failure of this result (here there are no more than $|\V|$, rather than the $2^{\aleph_0}$ mentioned in \cite{LiRo17}), is omitted here in the interest of space.





Let $\chi:X\to M$ be in $D_{\delta}$, meaning that $d(p_0\restr \chi,p_1\restr \chi)=\delta$. We show, first, that there exists a decreasing sequence $s_\chi=\langle \delta_n:n\in\omega\rangle $ such that $\delta=\bigwedge \{\delta_n: n\in\omega\}$ and $(f_0\circ \chi,f_1\circ \chi, a_0, a_1)$ belongs to the full image of $G_{\delta_n}$ for each $n$.  It is an immediate consequence of the sequential approximation property that there exists a decreasing sequence $\langle \delta^*_n: n\in\omega \rangle$ with $\delta^*_n\gg \delta$ such that $\delta=\bigwedge \{ \delta^*_n: n\in\omega\}$. Recall from the discussion above that we can obtain the distance $\delta=d(p_0\restr \chi,p_1\restr \chi)$ as the meet of the distances $d_N(Ug_0(a_0),Ug_1(a_1))$, where $g_i:{N}_i\to {N}$ and $g_0\circ (f_0 \circ \chi)= g_1\circ (f_1\circ \chi)$.  Since $\delta^*_n\gg\delta$, by constructively complete distributivity there exist $g^n_i:N_i\to N_n$, $i=0,1$, such that 
$$\delta^*_n\gg d_{N_n}(Ug^n_0(a_0),Ug^n_1(a_1))=\delta_n.$$ 
 
Let $\delta'=\bigwedge\{\delta_n:n\in\omega\}$. Because $\delta=d(p_0\circ \chi,p_1\circ\chi)$ is the meet of the distances $d_N(Ug_0(a_0),Ug_1(a_1))$, which include $d_{N_n}(Ug^n_0(a_0),Ug^n_1(a_1))=\delta_n$ for all $n\in\omega$, it must be the case that $\delta\le \bigwedge \{\delta_n: n\in\omega\}=\delta'$. On the other hand, since
$$\delta^*_n\gg d_{N_n}(Ug^n_0(a_0),Ug^n_1(a_1))=\delta_n\geq\bigwedge \{\delta_n:n\in\omega\}=\delta'$$
for all $n\in\omega$, $\delta'\le\bigwedge  \{\delta^*_n: n\in\omega\}=\delta$. Therefore $\delta=\delta'=\bigwedge \{\delta_n:n\in\omega\}$. Moreover, since $\delta_n=d_{N_n}(Ug^n_0(a_0),Ug^n_1(a_1))$, the pair 
$(f_0\restr \chi,f_1\restr \chi, a_0, a_1)$ belongs to the full image of $G_{\delta_n}$.

We may assume that there exists a cofinal set $D_s\subseteq D_{\delta}$ realizing the same sequence $\langle \delta_n:n\in\omega \rangle$. The argument, again by contradiction, is a slight variation on the one above.

Notice that $(f_0,f_1,a_0,a_1)$ can be obtained as the colimit of the $\kappa$-directed diagram of restrictions $(f_0\circ\chi,f_1\circ\chi,a_0,a_1)$, where $\chi:X\to M$, $X$ $\kappa$-presentable.  Since $D_s$ is cofinal in the colimit diagram for $M$, the colimit of $(f_0\circ\chi,f_1\circ\chi,a_0,a_1)$, $\chi\in D_s$, is precisely the same as that of the larger diagram; that is, $(f_0,f_1,a_0,a_1)$. Given that $G_{\delta_\E}$ is $\kappa$-accessible for all $\E\gg 0$, it is closed under $\kappa$-directed colimits; since $(f_0\circ\chi,f_1\circ\chi,a_0,a_1)\in G_{\delta_n}$ for all $\chi\in D_s$, $(f_0,f_1,a_0,a_1)$ belongs to the full image of $G_{\delta_n}$ for all $n\in\omega$. Therefore $d(p_0,p_1)\le \delta_n$ for all $n\in\omega$, so $d(p_0,p_1)\le \bigwedge \{\delta_n:n\in\omega\}=\delta$, meaning that $\E\gg d(p_0,p_1)$, as desired.
\end{proof}

\begin{remark}
	We note that the sequential approximation property, Definition \ref{defsflagg}(2), is not strictly necessary in this proof: by Lemma~\ref{Closeness}, we can always find a \emph{directed system} converging to $\delta$, which plays the role of the sequence in the proof above.  With this minor modification, and a great deal more bookkeeping, the argument proceeds just as above.
\end{remark}

It is important to note that Theorem~\ref{VTameness} goes beyond the analogous result for mAECs in \cite{LiRo17} in a number of respects, beyond replacing the reals with an arbitrarily strongly Flagg quantale: in particular, \cite[5.5]{LiRo17} assumes the CTP (Definition~\ref{CTP}) and that the underlying spaces are metric, rather than pseudometric.  While these conditions ensure that we have a metric---rather than a pseudometric---structure on the sets of Galois types (see Proposition~\ref{propctpmetr}), it transpires that this is not, in fact, necessary for the argument.

\section{Future directions}\label{secpartial}

For practical reasons, we have thus far worked with a dual formulation of quantales.  We now transform back to the more conventional formulation, laid out in Section~\ref{subsecquantprelims}: up to dualization in the quantale, any pseudometric of the kind described in Definition~\ref{defvmetr} corresponds to a map $E:M\times M\to\V^{op}$ satisfying the following conditions, among others:
\begin{enumerate}
	\item (Symmetry) $E(x,y)=E(y,x)$
	\item (Transitivity) $E(x,y)\wedge E(y,z)\le E(x,z)$
\end{enumerate}
If we restrict to the case in which the quantale is actually a frame, $\Omega$, this is precisely the definition of an {\it $\Omega$-valued set} (or {\it $\Omega$-set}), in the sense of \cite{fourmanscott}.  This forms the base for an important association: $\Omega$-sets are precisely the presheaves on $\Omega$, meaning that $\Omega$-pseudometric spaces can be thought of in this way as well.  Moreover, $\Omega$-metric spaces correspond to {\it separated $\Omega$-sets}, and thus to {\it separated presheaves on $\Omega$}.

There is a difficulty: in our $\Omega$-pseudometric spaces, we assume reflexivity; that is, $d(x,x)=0$ for all $x$.  In the case of $\Omega$-sets---or presheaves on $\Omega$---this means that $E(x,x)$ is the top element of $\Omega$ for any $x$, meaning that, as mentioned in Remark~\ref{rmkvsps}(2), all local sections of such a presheaf are global.  This suggests strongly that we have restricted ourselves too much, as locality is precisely what gives presheaf- and sheaf-theoretic semantics their added flexibility and power.  To avoid this problem, we must drop reflexivity as an axiom.  While this is a disturbing notion, it lands us in the realm of {\it partial metric spaces}, which have been well-studied in the context of, e.g., complexity in parallel computing (cf. \cite{mattflow}, \cite{wadgeflow}).  

Pending a careful study of completeness in partial $\Omega$-metric spaces---\cite{hofreisconv} is closely related---it should be possible to show that Cauchy complete partial $\Omega$-spaces correspond to {\it complete $\Omega$-sets}, and hence to {\it complete presheaves on $\Omega$}.  That is, they are {\it sheaves on $\Omega$}.  Simple computations reveal that nonexpanding maps of complete partial $\Omega$-spaces are precisely the natural transformations between the corresponding sheaves, so we in fact have a proper equivalence of categories.  This means that any results for the partial analog of $\Omega$-AECs in the sense of this paper will translate immediately to the context of structures on sheaves (or, if you prefer, sheaves of constant structures).  The result on the consistency of tameness, Theorem~\ref{VTameness}, requires total monotonicity, which is too much to ask in a general frame (recall Example~\ref{exsflagg}(4)), but this should still be a fruitful connection.

Although we will not dwell on this point, \cite{fuzzysetsi} and \cite{fuzzysetsii} illustrate the equivalence between fuzzy sets and sheaves, meaning that our results should transfer immediately to this context as well.

\bibliographystyle{spbasic}      

\end{document}